\documentclass[preprint]{elsarticle}

\newcommand\ifhidecomments\iftrue

\usepackage[utf8]{inputenc} 
\usepackage[T1]{fontenc}    
\usepackage{hyperref}       
\usepackage{url}            
\usepackage{booktabs}       
\usepackage{amsfonts}       
\usepackage{nicefrac}       
\usepackage{microtype}      
\usepackage{xcolor}         

\usepackage{natbib}
\usepackage{geometry}

\usepackage[utf8]{inputenc}
\usepackage{amsfonts}
\usepackage{amsmath}
\usepackage{amssymb}
\usepackage{appendix}
\usepackage{float}
\usepackage{amsthm}
\usepackage{enumitem}
\usepackage{graphicx}
\usepackage{tikz}
\usepackage{pgfplots}
\usepgfplotslibrary{fillbetween}
\usepackage{xcolor}
\usepackage{thmtools}
\usepackage{bm}
\usepackage{thm-restate}
\usepackage{wrapfig}
\usepackage{slashed}
\usepackage{comment}
\usepackage[capitalize]{cleveref}

\newcommand\RR{\mathbb{R}}
\newcommand\CC{\mathbb{C}}
\newcommand\EE{\mathbb{E}}
\newcommand\NN{\mathbb{N}}

\newcommand\cN{\mathcal{N}}

\newcommand\relu{{\sigma}}

\newcommand\opn[1]{\operatorname{#1}}
\newcommand\AS{\mathcal{A}}
\newcommand\AP{\mathcal{P}}

\newcommand\sub[1]{_{\opn{#1}}}

\newcommand\submax{\sub{max}}
\newcommand\rank{\opn{rank}}

\newcommand\softplus{\opn{softplus}}

\newcommand\activation{\sigma}

\newcommand\texttype[1]{\mathrm{#1}}

\newcommand\entwise{\odot} 
\newcommand\entp{\entwise} 

\newcommand\Lt{\mathcal{L}^2}
\newcommand\Ltd{\Lt(\RR^d;\Xdistr)}
\newcommand\Ltnd{\Lt(\RR^{nd};\Xdistr_n)}

\newcommand\ii{i}

\newcommand\rhoB{\rho}
\newcommand\Bsp{\mathfrak B}
\newcommand\ABsp{\mathfrak A}

\newcommand\Bnorm[1]{\|#1\|_{\Bsp}}
\newcommand\ABnorm[1]{\|#1\|_{\AS}}
\newcommand\PBnorm[1]{\|#1\|_{\AP}}

\newcommand\Bnormti[1]{\|#1\|_{\Bsp'}}

\newtheorem{theorem}{Theorem}
\newtheorem*{theorem*}{Theorem}
\newtheorem*{corollary*}{Corollary}
\newtheorem{lemma}[theorem]{Lemma}
\newtheorem{corollary}[theorem]{Corollary}

\newtheorem*{proposition*}{Proposition}

\newtheorem{definition}[theorem]{Definition}

\newtheorem{remark}[theorem]{Remark}

\renewcommand\i{i}
\renewcommand\v{\mathbf{v}}
\newcommand\w{\mathbf{w}}
\newcommand\x{\mathbf{x}}
\newcommand\y{\mathbf{y}}
\renewcommand\a{\mathbf{a}}
\renewcommand\aa{\mathcal{E}}
\renewcommand\b{\mathbf{b}}
\renewcommand\iota{i}
\newcommand\sigmaa{\boldsymbol{\sigma}}
\newcommand\e{e}
\newcommand\ee{\mathbf{e}}

\newcommand\xdist{\nu}
\newcommand\Xdist{\nu}
\newcommand\Xdistr{\Xdist}
\newcommand\Xnorm[1]{\|#1\|_{\Xdistr}}

\newcommand\bracket[2]{\langle#1|#2\rangle_{\Xdistr}}

\newcommand\LP[2]{{#1}^{\texttype{LP}(#2)}}
\newcommand\HP[2]{{#1}^{\texttype{HP}(#2)}}

\newcommand\ASF{{\mathcal F_\wedge}}

\newcommand\diag{\operatorname{diag}}

\ifhidecomments

\newcommand\NA[1]{}
\newcommand{\LL}[1]{}

\newenvironment{provisional}{\PackageError{no package}{unexpected provisional environment in final version}{resolve provisional section}}{}
\newcommand\shortprovisional[1]{\PackageError{no package}{unexpected provisional environment in final version}{resolve provisional section}}

\else

\usepackage{soul}
\newcommand\NA[1]{\textcolor{red}{\hl{[\textbf{NA:} #1]}}}
\newcommand\LL[1]{\textcolor{blue}{\hl{[\textbf{LL:} #1]}}}

\newcommand\shortprovisional[1]{{\color{gray}#1}}

\fi

\begin{document}

\title{
Anti-symmetric Barron functions and their approximation with sums of determinants
}

\author[1,2]{Nilin Abrahamsen\corref{cor1}}
\ead{nilin@berkeley.edu}

\author[1,3]{Lin Lin}
\ead{linlin@math.berkeley.edu}

\cortext[cor1]{Corresponding author}

\address[1]{
Department of Mathematics,
University of California, Berkeley, CA 94720 USA
}
\address[2]{
The Simons Institute for the Theory of Computing,
Berkeley, CA 94720 USA
}
\address[3]{
Applied Mathematics and Computational Research Division, Lawrence Berkeley National Laboratory, Berkeley, CA 94720, USA
}

\begin{abstract}
A fundamental problem in quantum physics is to encode functions that are completely anti-symmetric under permutations of identical particles. The Barron space consists of  high-dimensional functions that can be parameterized by infinite neural networks with one hidden layer. By explicitly encoding the anti-symmetric structure, we prove that the anti-symmetric functions which belong to the Barron space can be efficiently approximated with sums of determinants. This yields a factorial improvement in complexity compared to the standard representation in the Barron space and provides a theoretical explanation for the effectiveness of determinant-based architectures in ab-initio quantum chemistry.

\end{abstract}

\maketitle

\section{Introduction}
To simulate a physical system it is essential to construct a model which respects the \emph{symmetries} of the real-world problem. A prominent example is when a function is defined on \emph{sets} of points \cite{zaheer_deep_2017,santoro_simple_2017,yarotsky_universal_2022}, in which case the function can be viewed as a \emph{permutation-invariant} function of an input vector. A symmetry need not mean that a function is invariant to transformations of its input; more generally it can map transformations of the input to transformations of the output through a group homomorphism. \emph{Equivariance} is one such example where a transformation of the input gives rise to the same transformation on the output. Another such symmetry is \emph{anti-symmetry} where a permutation of the input vector multiplies the output by the \emph{sign} of the permutation.

Accurate modeling of fermionic systems is one of the most challenging and interesting problems in science. For example, the solution of the Schr\"odinger equation underlies all chemical properties of a given atomic system. 
Due to the Pauli exclusion principle, the fermionic wavefunction is anti-symmetric with respect to particle exchange. 
When the number of fermions grows, effective parametrization of such  wavefunctions can become increasingly difficult for many systems of interest. 
Anti-symmetric functions also arise in other contexts in machine learning such as determinantal point processes \cite{kulesza_determinantal_2012} where they are used to ensure diverse samples.

In the past decade there has been an explosive growth of techniques using neural networks (NN) as universal function approximators. The practical applicability of NNs is brought about by new software tools, hardware optimizations, as well as improved algorithms. NN approximators also significantly broaden the parameterization class for anti-symmetric functions in quantum physics~\cite{LuoClark2019,HanZhangE2019,HermannSchaetzleNoe2020,PfauSpencerMatthewsEtAl2020,StokesMorenoPnevmatikakisEtAl2020,LinGoldshlagerLin_vmcnet}. The combination of NN with variational Monte Carlo (VMC) methods provides a new path towards a low-scaling, systematically improvable method to approach the exact solution.

Despite recent progresses it is unclear how to construct a universal NN representation of anti-symmetric functions that does not obviously suffer from the curse of the dimensionality~\cite{LuoClark2019,PfauSpencerMatthewsEtAl2020,Hutter2020,HanLiLinEtAl2019}.
In the absence of symmetry constraints, very simple NN structures such as an NN with one hidden layer (sometimes also referred to as a ``two-layer'' NN) is already a universal function approximator~\cite{Cybenko1989,HornikStinchcombeWhite1989,Barron1993}.
Therefore in principle, explicitly antisymmetrizing a NN with one hidden layer can parameterize  universal anti-symmetric functions. 
One obvious drawback of this strategy is that the computational cost of the antisymmetrization step still increases factorially with respect to the system size. 
Nonetheless, such an explicitly anti-symmetrized NN structure has been recently studied in VMC calculations, which can yield effectively the exact ground state energy for small atoms and molecules~\cite{LinGoldshlagerLin_vmcnet}.

Conversely, determinant-based NN constructions of anti-symmetric functions can be evaluated efficiently. But their expressive power is unclear except in the setting of a combinatorially large number of determinants spanning the entire anti-symmetric subspace. It is therefore prudent to know if the efficient determinant-based constructions are able to capture the a priori intractable class of explicitly anti-symmetrized neural networks. 

\subsection{Contribution}
The \emph{Barron space}, defined in \cite{e_barron_2022} based on the seminal work of Barron \cite{Barron1993}, characterizes functions that can be approximated by an infinite neural network with one hidden layer. We consider the subspace of antisymmetric functions in the Barron space and prove (\cref{mainthm}):
\begin{enumerate}

\item A function in the anti-symmetric Barron space can be efficiently approximated using a sum of determinants.

\item The theoretical error bound factorially improves the error estimate in the standard Barron space.

\end{enumerate}
The Fourier transform is central to our analysis. 
This is because anti-symmetrizing a complex-valued plane wave gives rise to a determinant (called a Slater determinant). Each plane wave can the viewed a single hidden neuron with an exponential activation function

\subsection{Background and related works}

Consider a system of $n$ \emph{indistinguishable particles} in a $d$-dimensional space $\Omega\subset\RR^d$ ($d=1,2,3$), and let $N=nd$. The $n$-particle wave function is defined on inputs $\x=(x_1,\ldots,x_n)\in\Omega^n\subset\RR^N$ where each $x_i$ is in $\Omega\subset\RR^d$.
Indistinguishability means that  $\psi(x_1,\ldots,x_n)$ satisfies permutation symmetry of the norm $|\psi(x_1,\ldots,x_n)|$ under interchange of the $n$ inputs $x_i\in\Omega$. \emph{Fermions} are indistinguishable particles which satisfy the \emph{Pauli exclusion principle} and correspond to an \emph{anti-symmetric} wave function $\psi$. Anti-symmetry means that for a permutation $\pi\in S_n$, whose sign we denote by $(-1)^\pi$,
\[\pi(\psi)=(-1)^\pi\psi,\]where we have defined $\pi(\psi):\RR^{nd}\mapsto\CC$ by $\pi(\psi)(\x):=\psi(x_{\pi^{-1}(1)},\ldots,x_{\pi^{-1}(n)})$ for $x\in\RR^{nd}$. Let $\ASF_n$ denote the set of anti-symmetric functions $(\RR^d)^n\to\CC$.

In the machine learning literature there is a rich body of works related to permutation-invariant data, i.e., when the input data is a set \cite{zaheer_deep_2017,santoro_simple_2017,yarotsky_universal_2022,zweig_functional_nodate}. 
A widely used class of Ansatzes for anti-symmetric functions takes the form of a \textit{sum of Slater determinants}.  A Slater determinant, denoted $\phi_1\wedge\cdots\wedge\phi_n$, is constructed through $n$ \emph{orbitals}, i.e., functions $\phi_1,\ldots,\phi_n:\RR^d\to\CC$. The Slater determinant is the function defined by $(\phi_1\wedge\cdots\wedge \phi_n)(x_1,\ldots,x_n)=\det[E(\x)]$, where $(E(\x))_{ij}=\frac1{\sqrt{n!}}\phi_j(x_i),i,j=1,\ldots,n$. 

The representation of anti-symmetric functions is extensively studied in physics, but the literature on anti-symmetrized neural networks is sparse. 
Slater determinants can span a dense subset of the anti-symmetric space but the representation is very inefficient.  
Indeed, even in the case of a finite single-particle state space $|\Omega|=O(n)$ we would require $\binom{|\Omega|}{n}$ Slater determinants to span the anti-symmetric space. 
\cite{zweig_towards_2022} finds certain anti-symmetric functions that cannot be efficiently approximated using a simple sum of Slater determinants, but can be effectively expressed using a more complex Ansatz called the Slater-Jastrow form.

The FermiNet \cite{PfauSpencerMatthewsEtAl2020} and PauliNet \cite{HermannSchaetzleNoe2020} Ansatz have significantly expanded the representation power of the sum of determinants, by composing the orbitals with an equivariant mapping, and by parameterizing both using neural networks. The resulting structure can be expressed as a sum of generalized determinants
\begin{equation}
\psi(\x)=\sum_{k=1}^m \det(E_k(\x)),
\label{eq:composite}
\end{equation}
where $E_k:\RR^{nd}\to\RR^{n\times n}$ is equivariant, meaning that $E_k(\pi\x)=\pi(E_k(\x))$ and the permutation $\pi$ only exchanges rows of $E_k$. 

While the representation power of the Ansatz of the form \eqref{eq:composite} remains unclear, \cite{PfauSpencerMatthewsEtAl2020} provides an argument that with a sufficiently general mapping $E$, it is sufficient to choose $m=1$ to represent \textit{any} anti-symmetric function $\psi$. We recall their argument below.

\noindent{Proof of universality from \cite{PfauSpencerMatthewsEtAl2020}:}
\textit{Introduce an ordering $\le$ on vectors in $\RR^d$, for example a dictionary ordering on the $d$ coordinates. Given $\x\in\RR^{nd}$ let $\pi$ be the permutation such that $\pi^{-1}\x$ is sorted. That is, $\x=(\tilde x_{\pi(1)},\ldots,\tilde x_{\pi(n)})$ for some sorted $\tilde x_1\le\ldots\le\tilde x_n$. Then for any anti-symmetric function $\psi$, it is sufficient to choose $E(\x)=\pi\tilde\y$ where $\tilde\y=\diag((-1)^\pi\psi(\x),1,\ldots,1)$.}

The argument above has the drawback that the mapping $E$ is highly discontinuous. In this work we therefore aim to approximate a more restricted class of anti-symmetric functions with an Ansatz which is continuous with respect to $\x$ and obtain a quantizative error bound relative to a known complexity measure.

\subsection{Setup}
Based on the seminal work of Barron \cite{Barron1993}, 
the \emph{Barron space} is defined in \cite{e_barron_2022} as those functions which can be approximated by a continuous generalization of neural networks with one hidden layer. 

\begin{definition}[Barron space and norm \cite{e_barron_2022}]
\label{def:Barronspace}
The \emph{Barron space} $\Bsp$ is the set of functions of the form
    \begin{equation}\label{Bspdef}f_\rhoB(\x)=\int a\relu(\w\cdot \x+b)d\rhoB(a,b,\w),\end{equation}
    where $\relu(x)=\max\{0,x\}$ is the ReLU activation function and $\rhoB$ is a finite measure. 
 \cite{e_barron_2022} defines \emph{Barron norm} of $f$ as
    \begin{align}\label{eq:Barronfn}
    \Bnormti{f}&=\inf\{\tilde\varphi(\rho)\:|\:f_\rho=f\},\text{ where}\\
    \tilde\varphi(\rho)&=\int|a|(\|\w\|_1+|b|)d\rhoB(a,b,\w).
    \end{align}
    We further define the {translation-invariant} Barron norm $\Bnorm{f}\le\Bnormti{f}$ by
    \begin{align}
    \Bnorm{f}&=\inf\{\varphi(\rho)\:|\:f_\rho=f\},\text{ where}\\
    \varphi(\rho)&=\int|a|\|\w\|_1d\rho(a,b,\w).\label{varphidef}
    \end{align}
\end{definition}
We call $\rho$ a \emph{Barron measure} for $f$ if $f_\rho=f$.
We state our upper bound in terms of the translation-invariant Barron norm $\Bnorm{f}$ which implies the same bound relative to the larger norm $\Bnormti{f}$.

\section{Main result}
\label{sec:results}
We state our approximation result in terms of an upper bound on the error in the norm on $L^2(\Omega^n,\xdist^{\otimes n})$ where $\xdist=\cN(0,I)$ is a standard Gaussian envelope function. The interpretation is that the actual wave function is a normalized function $\Psi\in L^2(\Omega^n)$ which we represent as $\Psi(\x)=\sqrt\Xdist(\x)\psi(\x)$ since $\Psi$ is localized. Writing $\tilde\Psi(\x)=\sqrt\Xdist(\x)\tilde\psi(x)$ we then have that $\Xnorm{\tilde\psi-\psi}=\|\tilde\Psi-\Psi\|$ is the standard $L^2$-distance between the wave functions $\tilde\Psi$ and $\Psi$.

\begin{theorem}
\label{mainthm}
    Let $\psi$ be an antisymmetric function which belongs to the Barron space and has translation-invariant Barron norm $\Bnorm\psi$ (\cref{def:Barronspace}). 
Then for each $m\in\NN$ there exists a linear combination $\psi_m=\frac1m\sum_{k=1}^m \mu_k\aa_{\w_k}$ of $m$ Slater determinants of the form $\aa_{\w}=e_{w_1}\wedge\cdots\wedge e_{w_n}$ with $e_w(x)=e^{\ii w\cdot x}$ such that 
    \begin{equation}
    \Xnorm{\psi_m-\psi}\le\frac{\Bnorm{\psi}}{\sqrt{n!}}\Big(\frac{C}{\sqrt m}+2^{-\Omega(n)}\Big).
    \end{equation}
     Here, $C=2\sqrt{d}/\pi$. In particular $\psi_m$ is of the form \cref{eq:composite} with $m$ determinants. 
\end{theorem}

It was previously known that a Barron function $f$ can be approximated by \emph{finite} neural networks with one hidden layer of $m$ neurons up to error $\Bnorm{f}/\sqrt m$ \cite{Barron1993,e_barron_2022}. Our determinant-based approximation in \cref{mainthm} improves factorially on this estimate by a factor $1/\sqrt{n!}$ in the anti-symmetric setting.
This illustrates that the approximation is highly inefficient if the anti-symmetry condition is not explicitly built into the Ansatz.

\cref{mainthm} motivates the following definition:

\begin{definition}[Anti-symmetric Barron space and norm]\label{def:Anorm}
    The anti-symmetric Barron space is the subspace of the Barron space consisting of anti-symmetric functions, that is,
    \begin{equation}
    \ABsp=\Bsp\cap\ASF.
    \end{equation}
    For $\psi\in\ABsp$ we define its \emph{anti-symmetric Barron norm} as
    \begin{equation}\label{eq:Anorm}
    \ABnorm{\psi}=\frac{\Bnorm{\psi}}{\sqrt{n!}}.
    \end{equation}
\end{definition}

We can then restate our result as follows:
\begin{corollary}
\label{maincor1}
For any $\psi\in\ABsp$ and each $m\in\NN$ there exists a linear combination $\psi_m=\frac1m\sum_{k=1}^m \mu_k\aa_{\w_k}$ of $m$ Slater determinants $\aa_{\w}$ such that 
    \begin{equation}
    \Xnorm{\psi_m-\psi}\le\ABnorm{\psi}\Big(\frac{C}{\sqrt m}+2^{-\Omega(n)}\Big),
    \end{equation}
    where $\ABnorm{\psi}$ is given by \cref{eq:Anorm}.
\end{corollary}

\citet[Section IX]{Barron1993} provided a number of examples of Barron functions. 
Some care must be taken when restricting these to the anti-symmetric case. For example, any radial function $f(\x)=g(|\x|)$, or any function $f(\x)$ that is symmetric with respect to two of its coordinates vanishes after anti-symmetrization. This issue can be overcome by applying a translation by some vector $(x_1,\ldots,x_n)\in\RR^{nd}$ with distinct components $x_i\neq x_j$ before anti-symmetrizing. 
More general anti-symmetric Barron functions can be constructed from anisotropic ridge functions $f(\x)=g(\a\cdot \x)$, anisotropic radial functions $f(\x)=g(|A\x|)$, and anisotropic integral representations $f(\x)=\int K(\a\cdot \x+b) d\rhoB(\a,b)$, to name a few.

A Slater determinant $\psi=\phi_1\wedge\cdots\wedge\phi_n$ can be written as $\psi=\sqrt{n!}\AP(\phi_1\otimes\cdots\otimes\phi_n)$ where $\AP$ is the projection onto the subspace of anti-symmetric functions. It therefore follows that \[\ABnorm{\phi_1\wedge\cdots\wedge\phi_n}=\Bnorm{\phi_1\otimes\cdots\otimes\phi_n}.\]

\section{Proof sketch}
We now outline the proof of \cref{mainthm}.
We can state the property of being antisymmetric as $\psi=\AP\psi$ where $\AP$ is the projection onto $\ASF$. Given a basis expansion $\psi=\int f_w d\mu(w)$ of $\psi\in\ASF$ we will take the projection inside the integral and write $\psi=\int\AP(f_w) d\mu(w)$. We therefore need a basis expansion such that:
\begin{enumerate}
\item\label{it:easyP} We can analyze anti-symmetric projection of $f_w$ and estimate the magnitude of $\AP(f_w)$. 
\item\label{it:B_to_exp} The expansion $\rho$ of a Barron function $\psi_\rho$ (\cref{def:Barronspace}) gives rise to a basis expansion $\mu(w)$ into functions $f_w$.
\end{enumerate}
We will show that the Fourier transform provides such an expansion.

The Fourier basis functions are complex plane waves $\x\mapsto e^{i\w\cdot\x}$.
To analyze their anti-symmetrization (\cref{it:easyP}), observe that they factor into a product:
\begin{equation}
\ee_{\w}(\x):=\e(\w\cdot\x)=e^{\i\w\cdot\x}=\prod_{j=1}^ne^{\i w_j\cdot x_j}.
\end{equation}
where $\i\in\CC$ is the complex unit. 
The projection of $\ee_\w$ onto $\ASF$ can then be computed as a Slater determinant. Specifically, 
\begin{align}
\label{eq:Pew}
\AP\ee_{\w}=\AP(\otimes_{j=1}^n e_{w_j})=\frac1{\sqrt{n!}}\aa_\w,
\end{align}
where $\aa_\w=e_{w_1}\wedge\cdots\wedge e_{w_n}$ is as in \cref{mainthm}.
To obtain \cref{it:B_to_exp} we use the Barron expansion \cref{eq:Barronfn} of a function $\psi\in\Bsp$ to obtain its Fourier decomposition. Concretely, \cref{eq:Barronfn} given a decomposition of $\psi$ into ridge functions 
\begin{equation}\label{eq:ridgefn}
\sigmaa_{\w,b}(\x)=\sigma(\w\cdot\x+b).
\end{equation}
We then apply the one-dimensional Fourier decomposition of ReLU to decompose the ridge function into Fourier basis functions on $\RR^{nd}$.

The remainder of this sketch is a formal derivation which will require additional work in the following sections to be made rigorous. Assume that the activation function satisfies the Fourier inversion formula:
\begin{equation}
\label{formal0}
\sigma(y)=\frac1{\sqrt{2\pi}}\int \hat\sigma(\theta)e^{i\theta y}d\theta.
\end{equation}
This is not immediately well-defined in the case of ReLU because $\hat\sigma$ is not absolutely integrable. Substituting $\w\cdot\x+b$ into \cref{formal0} yields a decomposition of the ridge function $\sigmaa_{\w,b}$ on $\RR^{nd}$. By \cref{eq:Pew}, projecting this ridge function onto the anti-symmetric subspace yields
\begin{equation}
\label{ASridge}
\AP\sigmaa_{\w,b}=\frac1{\sqrt{2\pi n!}}\int\hat\sigma(\theta)e^{ib\theta}\aa_{\theta\w}d\theta.
\end{equation}
The anti-symmetric Barron function $\psi$ is of the form $\psi=f_\rho$ for some measure $\rho$. We antisymmetrize the integral representation of $f_\rho$ and expand the anti-symmetrized ridge function as in \cref{ASridge} to obtain
\begin{align}
f_\rho=\AP f_\rho
&=\int a\,\AP(\sigmaa_{\w,b})\,d\rho(a,b,\w)\\
&=\frac1{\sqrt{2\pi n!}}\iint a\hat\sigma(\theta)e^{ib\theta}\aa_{\theta\w}d\rho(a,b,\w),
\label{bigint}
\end{align}
which is an expansion as an integral over the basis functions $\aa_{\theta\w}$ against a complex measure $\mu$. By a standard sampling argument this can be approximated up to error $\|\mu\|/\sqrt m$ with a finite sum of $m$ terms, where $\|\mu\|$ is the total variation of the measure.

\begin{figure}
    \centering
    \includegraphics[width=.9\textwidth]{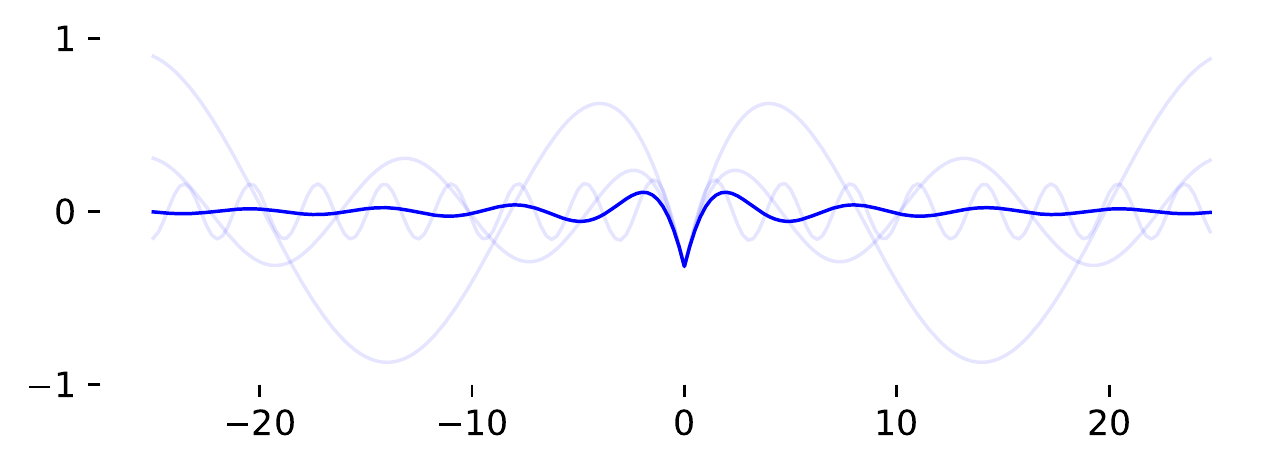}
    \caption{High-passed version $\HP\relu\gamma$ of ReLU (\cref{uvrelu}) for thresholds $\gamma=1\text{ (opaque)}$ and $\gamma=\frac14,\frac12,2$ (faint blue, with $\gamma=1/4$ being the large and slowly oscillating curve and $\gamma=2$ being the fast-oscillating curve). 
    We use anti-symmetrized ridge functions $\AP(\tilde\sigmaa_{\w,b})$ defined with the activation function $\tilde\sigma=\HP\relu\gamma$ to approximate anti-symmetrized ridge functions $\AP(\sigmaa_{\w,b})$ defined with the ReLU activation. 
    }\label{fig:hp}
\end{figure}

The formal identities \cref{formal0}--\cref{bigint} do not directly apply for ReLU due to a divergence at $\theta\to0$. To overcome this subtlety we decompose the ReLU activation function into a \emph{high-passed} or \emph{ultraviolet} part and a \emph{low-passed} or \emph{infrared} remainder. 
Specifically, the ultraviolet part of the ReLU activation function is
\begin{equation}\label{uvrelu}
\HP\activation{\gamma}(y)=
|y|/2-\frac{\cos(\gamma y)}{\pi \gamma }-\frac{y\opn{Si}(\gamma y)}\pi,\end{equation}
where $\opn{Si}(y)=\int_0^y\frac{\sin s}sds$. 

We prove asymptotic bounds on $\Xnorm{\aa_\w}$ for small $\w$ which show that the contributions from the infrared remainder are exponentially small after anti-symmetrization. We can therefore truncate away the infra-red part to avoid the divergence at small $\theta$ at the cost of an exponentially small error term. This truncation is equivalent with replacing the ReLU activation by its high-passed part (\cref{fig:hp}). We emphasize that the magnitude of the discarded infrared remainder is not small as a one-dimensional function. Rather, its smoothness means as a multidimensional ridge function it is near-orthogonal to the antisymmetric subspace.

\section{A renormalized anti-symmetrization operator}
It is natural to renormalize the anti-symmetric projection to
\begin{equation}\label{ASdef}\AS f=\sqrt{n!}\:\AP f=\frac1{\sqrt{n!}}\sum_{\pi\in S_n}(-1)^\pi\pi(f).\end{equation}
In particular, if
$f=\phi_1\otimes\cdots\otimes \phi_n$ is a tensor product of single-particle orbitals, then $\psi(x)=(\AS \phi)(x)$ is the Slater determinant $\phi_1\wedge\cdots\wedge \phi_n$. The normalization in Eq. \eqref{ASdef} is such that if $\phi_i$ are orthonormal functions on $L^2(\Omega,\xdist)$ then $\psi=\AS f$ is normalized in $L^2(\Omega^{n},\xdist^{\otimes n})$. This follows from Pythagoras' theorem because orthogonality of $\psi_i$ implies that the $n!$ terms $\pi(f)$ in \cref{ASdef} are orthonormal.
\label{productslater}

With the renormalized antisymmetrization operator we have another equivalent definition of the antisymmetric Barron norm.
\begin{lemma}\label{ABequiv}
The \emph{anti-symmetric Barron space} is equal to $\ABsp=\AS\Bsp:=\{\AS f|f\in\Bsp\}$,
and the \emph{anti-symmetric Barron norm} of an anti-symmetric $\psi$ is
\begin{equation}\label{ABdef}
\begin{aligned}
\ABnorm{\psi}&=\inf\{\Bnorm{f}\:|\:\AS f=\psi\}.
\end{aligned}
\end{equation}
\end{lemma}
We include the straightforward derivation of \cref{ABequiv} in \ref{sec:otherproofs}.

\section{Generalized Fourier inversion formula}
\label{sec:proofs}

\label{divergentF}
Our proof of \cref{mainthm} uses the Fourier decomposition of $f_\rhoB$ which we characterize using the Fourier transform of the ReLU activation function. 
Since ReLU is not integrable its Fourier transform is not defined as a convergent integral but rather in the sense of \emph{tempered distributions} \cite{reed_i_1981}. It this sense, ReLU has the Fourier transform
\begin{equation}\widehat\relu(\theta)=\frac{-1}{\sqrt{2\pi}\cdot\theta^2}+\sqrt{\tfrac\pi2}i\delta'(\theta).\end{equation}

We will not need the precise definition of $\hat\sigma$ but only that it satisfies the following, which we term the \emph{ultraviolet Fourier inversion formula}: For $t>0$,
\begin{equation}
\sigma(y)=\frac1{\sqrt{2\pi}}\int_{|\theta|>t}\hat\sigma(\theta)e^{i\theta y}d\theta
+p_t(y)+O(t g(y)),
\label{Fdef}
\end{equation}
where $p_t$ is a polynomial whose degree is bounded (uniformly in $t$), and $g$ is a non-negative function bounded by a polynomial. 
In \ref{invformula} we show that ReLU satisfies \cref{Fdef} with
\begin{equation}\label{Frelu}\widehat\relu(\theta)=\frac{-1}{\sqrt{2\pi}\cdot\theta^2},\qquad|\theta|>0.\end{equation}
and with remainders $p_\gamma(y)=y/2+\frac1{\pi\gamma}$, $g(y)=y^2$.

To state the generalized Fourier inversion formula more compactly, define the high-frequency part of an activation function $\sigma$ as follows: 
\begin{definition}\label{def:hpdef}
For $\sigma:\RR\to\CC$ define its \emph{high-pass} or $\HP{\sigma}{\gamma}$ at threshold $\gamma>0$ by
\begin{equation}
\label{fourieractivation}
\HP{\sigma}{\gamma}(y)=\frac1{\sqrt{2\pi}}\int_{|\theta|>\gamma}\hat\sigma(\theta)e^{i\theta y}d\theta.
\end{equation}
We define the low-pass $\LP\sigma{\gamma}$ as the remainder $\sigma-\HP\sigma{\gamma}$. 
\end{definition}
Then the ultraviolet Fourier inversion formula \cref{Fdef} holds when the remainder $\sigma-\HP\sigma{\gamma}$ is of the form
\begin{equation}
\label{uvi}
\sigma-\HP\sigma{\gamma}=p_\gamma+O(\gamma g).\end{equation}
The high-pass of ReLU is \cref{uvrelu} as shown in (\ref{invformula}).

We will prove that we can replace $\sigma$ by its high-pass and that the error incurred becomes exponentially small after anti-symmetrization. 
As a first step towards this error bound, observe that for an anti-symmetric Barron function, the contribution from the term $p_\gamma$ in the Fourier inversion formula vanishes.
\begin{lemma}\label{lem:polynomial}
If $f:\RR^{nd}\to\CC$ is a polynomial of degree $\deg f\le n-2$, then $\AS f\equiv 0$. In particular $\AS\sigmaa_{\w,b}\equiv 0$ if $\sigma$ is an activation function which is a polynomial of degree $\deg\sigma\le n-2$.
\end{lemma}
\begin{proof}
By linearity it suffices to prove the claim when $f$ is a monomial $f(\x)=\prod_{i=1}^n x_i^{r_i}$ where $r_i\in\NN_0$. Since $\deg f=\sum_i r_i\le n-2$ there exists a pair $i\neq j$ such that $r_i,r_j=0$. Let $\pi_{ij}$ be the permutation which swaps $i$ and $j$. Then $f(\pi_{ij}(\x))=f(\x)$ because $f$ does not depend on $x_i,x_j$. But we also have $f(\pi_{ij}(\x))=-f(\x)$ by anti-symmetry, so $f(\x)=0$.
\end{proof}

We substitute $y=\w\cdot\x+b$ into the ultraviolet Fourier inversion formula \eqref{Fdef} to obtain a decomposition of ridge functions on $\RR^{nd}$
\begin{equation}\label{FdefRnd}
\begin{aligned}
\sigmaa_{\w,b}(\x)&=\frac1{\sqrt{2\pi}}\int_{|\theta|>\gamma}\hat\sigma(\theta)e^{i\theta b}e^{i\theta\w\cdot\x}d\theta
\\&+p_\gamma(\w\cdot\x)+O(\gamma g(\w\cdot\x)),
\end{aligned}
\end{equation}
We anti-symmetrize this ridge function and apply \cref{lem:polynomial} which yields that for $n\ge\deg p_\gamma+2$,
\begin{align}\label{FdefA}
\AS\sigmaa_{\w,b}&=\lim_{\epsilon\to0}\frac1{\sqrt{2\pi}}\int_{|\theta|>\epsilon}e^{ib\theta}\hat\sigma(\theta)\aa_{\theta\w}d\theta,
\end{align}
where convergence is in the $\Ltnd$-norm ($g(\w\cdot\x)$ is bounded in this norm because of the fast-decaying Gaussian envelope $\Xdist$). 

\section{Properties of the anti-symmetrized Fourier basis functions}

\cref{fig:Dtt} illustrates that $\Xnorm{\aa_\w}$ is bounded by $1$ and vanishes for small $\w$ (\cref{fig:Dtt}). \cref{onebound} and \cref{prop:detbound} below capture this fact rigorously.
\begin{figure}[h]
    \centering
    \includegraphics[width=.9\textwidth]{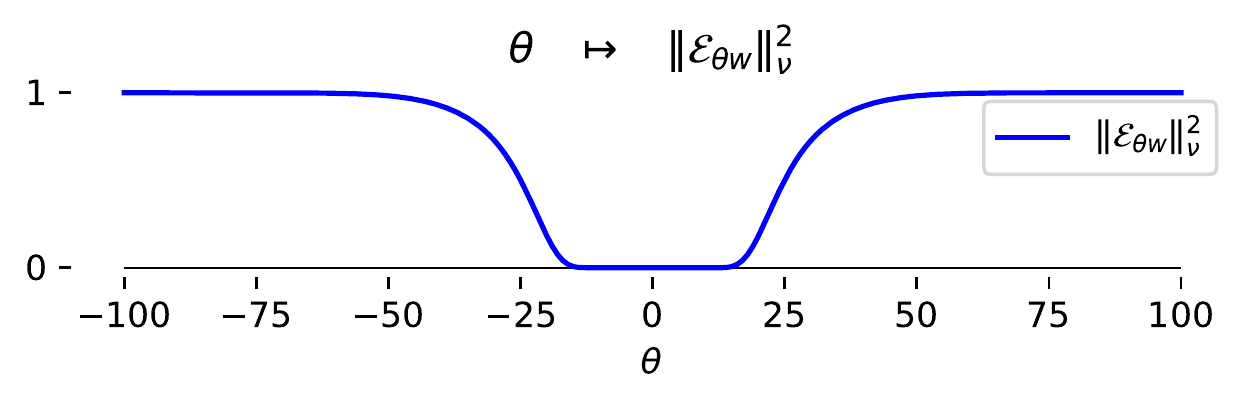}
    \caption{
    The function $\theta\mapsto\Xnorm{\aa_{\theta\w}}^2$ (for a randomly chosen $\w$ scaled to have $\|\w\|_\infty=1$). $n=100$, $d=3$.}\label{fig:Dtt}
\end{figure}
To analyze the behavior of $\Xnorm{\aa_w}$ we use that the overlap between two Slater determinants is the determinant of the \emph{overlap matrix} \cite{lowdin_quantum_1955}, meaning that 
\begin{equation}\label{AeAe}
    \bracket{\aa_\v}{\aa_\w}=\bracket{\wedge_i e_{v_i}}{\wedge_i e_{w_i}}=\det B^{(\v,\w)},
\end{equation}
where $B^{(\v,\w)}\in\RR^{n\times n}$ is given by
\begin{equation}
B^{(\v,\w)}_{ij}=\bracket{e_{v_i}}{e_{w_j}}.
\end{equation}
By \cref{AeAe} the problem of bounding the norms and overlaps of functions $\aa_\w$ corresponds to bounding the magnitude of a determinant. We begin with a simple uniform bound \cref{onebound} before proving the more technical $\w$-dependent bound (\cref{prop:detbound} below) which will lead to the exponentially small error term in \cref{mainthm}.

\begin{lemma}\label{onebound}
$\Xnorm{\aa_\w}\le1$ for all $\w\in\RR^{nd}$.
\end{lemma}
\begin{proof}
$B_{ij}^{(\w,\w)}$ is the Gram matrix of the $n$ vectors $e_{w_i}\in\Ltd$ so it is a positive semidefinite matrix. It then satisfies $\det(B^{(\w,\w)})\le\prod_i B_{ii}^{(\w,\w)}$ by Hadamard's theorem. So
\begin{align}\label{normsqB}
\Xnorm{\aa_\w}^2&=\det(B^{(\w,\w)})
\\&\le\prod_i B_{ii}^{(\w,\w)}=\prod_{i=1}\bracket{e_{w_i}}{e_{w_i}}^n=1.
\end{align}
\end{proof}

When the envelope is the standard Gaussian $\Xdistr=\cN(0,I_{nd})$ and $\v=\w$, the overlap matrix specializes to (letting $Y\sim\cN(0,1)$)
\begin{align}\label{GaussianD0}
B^{(\w,\w)}_{ij}
&=\EE[\overline{e^{iw_{i1}Y}}e^{iw_{j1}Y}]\cdots\EE[\overline{e^{iw_{id}Y}}e^{iw_{jd}Y}]
\\&=\EE[e^{i(w_{j1}-w_{i1})Y}]\cdots\EE[e^{i(w_{j1}-w_{i1})Y}]
\\&=e^{-\frac12\|w_j-w_i\|^2}.
\end{align}
We apply \cref{normsqB} and expand the square to obtain:
\begin{equation}\label{GaussianDdiag}
\Xnorm{\aa_\w}^2=e^{-\|\w\|^2}\det\big((e^{w_i\cdot w_j})_{ij}\big).
\end{equation}

\section{Determinant bound}
To obtain an upper bound on \cref{GaussianDdiag} we decompose the matrix $(e^{w_i\cdot w_j})_{ij}$ into a sum $\sum_{k=0}^\infty Q_k$, bounding the ranks and operator norms of the terms $Q_k$. For $L=\sum_{k=1}^{p-1}\rank Q_k$ we can then bound the $L$-th eigenvalue as the tail sum $\sum_{k=p}^\infty\|Q_k\|$. Taking the product of the eigenvalues yields a bound on the determinant and therefore on the norms of anti-symmetrized plane waves $\aa_\w$.

\begin{restatable}{lemma}{lemmalowranksum}
    \label{lem:rank_one_terms}
Let $\v=(v_1,\ldots,v_n)^T\in\RR^{n\times d}$ and $\w=(w_1,\ldots,w_n)^T\in\RR^{n\times d}$. Then $(e^{v_i\cdot w_j})_{ij}=\sum_{k=0}^\infty Q_k$ where
\begin{equation}\opn{rank}Q_k\le\binom{k+d-1}{d-1},\qquad\|Q_k\|\le \frac{n(\|v\|_\infty\|w\|_\infty d)^k}{k!}.\label{Qk}\end{equation}
\end{restatable}
\begin{proof}
Let $(c_1,\ldots,c_d)$ and $(\tilde c_1,\dots,\tilde c_d)$ be the columns of $v$ and $w$ and let $\entwise$ denote elementwise operations (exponentiation and product, respectively). Then,
\begin{equation}
    \label{columndecomp}
(e^{v_i\cdot w_j})_{ij}=e^{\entwise\sum_{i=1}^dc_i\tilde c_i^T}=\entp_{i=1}^d e^{\entwise c_i\tilde c_i^T}.
\end{equation}
We first consider each factor $e^{\entwise c_i\tilde c_i^T}$ separately. Elementwise multiplication of rank-one matrices given as outer products corresponds to elementwise multiplication of the vectors, $ab^T\entp \tilde a\tilde b^T=(a\entp\tilde a)(b\entp\tilde b)^T$. Therefore, applying the Taylor expansion entrywise,
\begin{equation}
    \label{singlecolumn}
e^{\entwise c\tilde c^T}
=
\sum_{k=0}^\infty\frac{(c\tilde c^T)^{\entwise k}}{k!}
=
\sum_{k=0}^\infty\frac{(c^{\entwise k})(\tilde c^{\entwise k})^T}{k!},
\end{equation}
where $c=c_i$, $\tilde c=\tilde c_i$ are column vectors. Apply \eqref{singlecolumn} to each factor of \eqref{columndecomp} and expand the sums,
\begin{align}
\entp_{i=1}^de^{\entwise c_i\tilde c_i^T}
&=
\sum_{k_1,\ldots,k_d=0}^\infty\frac{(\entp_{i=1}^d c_i^{\entwise k_i})(\entp_{i=1}^d \tilde c_i^{\entwise k_i})^T}{\prod_{i=1}^dk_i!}.
\\&=
\sum_{k=0}^\infty
\sum_{k_1+\ldots+k_d=k}^\infty\frac1{k!}\binom{k}{k_1,\ldots,k_d}(\entp_{i=1}^d c_i^{\entwise k_i})(\entp_{i=1}^d \tilde c_i^{\entwise k_i})^T\label{doublesum}
\end{align}
Let $Q_k$ be the innermost sum of \eqref{doublesum}. We estimate the maximum over the entries,  
\[\|Q_k\|\submax\le\frac{\|\v\|_\infty^{k}\|\w\|_\infty^k}{k!}\sum_{k_1+\ldots+k_d=k}^\infty\binom{k}{k_1,\ldots,k_d}=\frac{\|\v\|_\infty^k\|\w\|_\infty^kd^k}{k!}\]
and apply the inequality $\|Q_k\|\le n\|Q_k\|\submax$.  
\end{proof}

\begin{lemma}\label{lem:eigbound}
Let $\lambda_0\ge\lambda_1\ge\ldots$ be the absolute values of the eigenvalues of $(e^{v_i\cdot w_j})_{ij}$ and let $\mu=\|\v\|_\infty\|\w\|_\infty d$. Then $\lambda_0\le ne^\mu$, and for $\mu\le1/2$ and $p\in\NN$,
\begin{equation}
    \label{binomkd}
    \lambda_L\le\frac{2n}{p!}\mu^p,\qquad L=\binom{p+d-1}{d},
\end{equation}
where the case $p=0$ of \eqref{binomkd} holds with the interpretation $L=\binom{d-1}{d}=0$, $\lambda_0\le ne^{1/2}\le 2n$.
\end{lemma}
\begin{proof}
From the identity
\[
\binom{p+d-1}{d}=1+d+\binom{d+1}{d-1}+\cdots+\binom{p+d-2}{d-1},
\]
there are 
\[1+d+\binom{d+1}{d-1}+\cdots+\binom{p+d-2}{d-1}\ge\rank Q_0+\cdots+\rank Q_{p-1}\]
eigenvalues in front of $\lambda_L$ where $L=\binom{p+d-1}{d}$, and we have used \cref{lem:rank_one_terms}. By the min-max principle,
\[\lambda_L\le\left\|\sum_{k=p}^\infty Q_k\right\|\le n\sum_{k=p}^\infty\frac{\mu ^k}{k!}=\frac{n}{p!}\sum_{k=p}^\infty\mu^k=\frac{n}{p!}\frac{\mu^p}{1-\mu}\le \frac{2n}{p!}\mu^p.\]
\end{proof}

\begin{restatable}{proposition}{propdetbound}\label{prop:detbound}
Let $\gamma=\frac1{2\sqrt{d}}$ and let $p$ be any integer such that $\binom{p+d-1}{d}\le n/2$ and $p!\ge 4n^2$. Then,
\begin{equation}
\det((e^{w_i\cdot w_j})_{ij})\le \Big(\frac{\|\w\|_\infty}{2\gamma}\Big)^{pn}  \quad\text{ for }\quad  \|\w\|_\infty\le\gamma.
\end{equation}
\end{restatable}
\begin{proof}
Let $\v=\w$. By \cref{lem:eigbound} and the assumptions on $p$ we have $\lambda_{\lfloor n/2\rfloor}\le\frac{2n}{p!}\mu^p\le\frac{\mu^p}{2n}$ and $\lambda_0\le 2n$ where $\mu=d\|\w\|_\infty^2$, so it follows that 
$|\det((e^{w_i\cdot w_j})_{ij})|\le\lambda_0^{n/2}\lambda_{n/2}^{n/2}\le(\mu^p)^{n/2}=(\|\w\|_\infty\sqrt d)^{pn}=(\frac12\frac{\|\w\|_\infty}{\gamma})^{pn}$. This holds when $d\|\w\|_\infty^2=\mu\le1/4$, i.e., when $\|\w\|_\infty^2\le\gamma^2$.
\end{proof}

\cref{prop:detbound} suffices to give a fine-grained bound on the norms of the anti-symmetrized plane waves $\aa_\w$.
Let $p=\Theta(n^{1/d})$ and apply the bound to \cref{GaussianDdiag} to obtain:
\begin{equation}\label{Dttlowbound}
\Xnorm{\aa_{\w}}^2\le\Big(\frac{\|\w\|_\infty}{2\gamma}\Big)^{\Omega(n^{1+1/d})}
\end{equation}
for $\|\w\|_\infty\le\gamma$, where $\gamma=\frac1{2\sqrt d}$.

\section{Bound on the infra-red truncation error}
In the limit as the infra-red cutoff $\epsilon\to0$, \cref{FdefA} provides an expansion of an anti-symmetrized ridge function into functions $\aa_\w$. We need to bound the error when evaluating \cref{FdefA} with a finite infra-red truncation. We denote the anti-symmetric functions defined with or without such a truncation as follows:
\begin{definition}
Let $A_{\w,b}=\AS\sigmaa_{\w,b}$, and for $\gamma>0$, let
\begin{equation}
A_{\w,b;\gamma}=\frac1{\sqrt{2\pi}}\int_{|\theta|\ge\gamma}e^{ib\theta}\hat\sigma(\theta)\aa_{\theta\w}d\theta.
\label{At}
\end{equation}
To bound the error incurred from the infrared truncation we use the following triangle inequality:
\end{definition}
\begin{lemma}[Triangle inequality]
\begin{equation}
\Xnorm{A_{\w,b;\gamma}-A_{\w,b}}\vphantom{\frac11}
\le
\frac1{\sqrt{2\pi}}\int_{0<|\theta|<\gamma}|\hat\sigma(\theta)|\Xnorm{\aa_{\theta\w}}d\theta,
\end{equation}
where the right-hand side is interpreted as the limit \cref{limeps} below.
\end{lemma}
\begin{proof}
We use \cref{FdefA} to write $A_{\w,b}$ as a limit. Then,
\begin{align}\label{triangle}
&\Xnorm{A_{\w,b;\gamma}-A_{\w,b}}\vphantom{\frac11}
\\=&\lim_{\epsilon\to0}\Xnorm{A_{\w,b;\gamma}-A_{\w,b;\epsilon}}\vphantom{\frac11}
\\\le&\lim_{\epsilon\to0}\frac1{\sqrt{2\pi}}\int_{\epsilon<|\theta|<\gamma}|\hat\sigma(\theta)|\Xnorm{\aa_{\theta\w}}d\theta\label{limeps}
\end{align}
\end{proof}
Combining the triangle inequality with the bound from \cref{prop:detbound} yields the following error bound, as shown in \ref{ests}.
\begin{corollary}\label{AAbound}
Suppose $\|\w\|_\infty=1$ and let $\gamma=\frac1{2\sqrt{d}}$. Then,
\begin{equation}
\Xnorm{A_{\w,b;\gamma}-A_{\w,b}}\le 2^{-\Omega(n^{1+1/d})}.
\end{equation}
\end{corollary}

\section{Expanding a Barron function}


Given an anti-symmetric Barron function $\psi$, let $\rho$ be a Barron measure, i.e.,
\begin{equation}\label{fullexpand}
\psi=\AS f_\rho=\int a \: A_{\w,b} \: d\rho(a,b,\w).
\end{equation}
We say that $\rho$ is \emph{canonical} if $\|\w\|_\infty=1$ for all $(a,b,\w)$ in the support of $\rho$. As the next lemma shows we may assume without loss of generality that $\rho$ is canonical.

\begin{lemma}
Fix $p\in[1,\infty]$. In the definition of the Barron norm we may restrict to measures $\rho$ such that $\|\w\|_p=1$ for all $(a,b,\w)$ in the support of $\rhoB$. The resulting definition is equivalent with the original one.
In particular,
\begin{equation}
\ABnorm\psi=\inf\{\varphi(\rho)\:|\:\rho\text{ is canonical and }\AS f_\rho=\psi\}.
\end{equation}
\end{lemma}
\begin{proof}
Given $\rho$ such that $f_\rho=f$, define $\tilde\rho=h(\rho)=\rho\circ h$ where $h(a,b,\w)=(\tilde a,\tilde b,\tilde\w)$ where
\[\tilde a=\|\w\|_p a,\quad
\tilde b=b/\|\w\|_p,\quad
\tilde\w=\w/\|\w\|_p.\]
Then $\varphi(\tilde\rho)=\varphi(\rho)$, and $f_{\tilde\rho}=f_\rho$ due to the homogeneity of ReLU.
\end{proof}
For any canonical $\rho$ we define
\begin{equation}\label{truncexpand}
\psi_\gamma^{(\rho)}=\int a\:A_{\w,b;\gamma}\:d\rho(a,b,\w).
\end{equation}
We apply the usual triangle inequality to the integral over $\rho$ and then apply \cref{AAbound} to obtain
\begin{align}\label{prelimit}
\Xnorm{\psi_\gamma^{(\rho)}-\psi}
&\le\int |a|\Xnorm{A_{\w,b;\gamma}-A_{\w,b}} d\rho(a,b,\w)\\
&=2^{-\Omega(n^{1+1/d})}\int |a| d\rho(a,b,\w)\\
&=2^{-\Omega(n^{1+1/d})}\varphi(\rho),
\end{align}
where $\gamma=\frac1{2\sqrt{d}}$.
By expanding $A_{\w,b;\gamma}$ we can write $\psi_\gamma^{(\rho)}$ in the following form:
\begin{definition}\label{defcomplexmeasure}
Given a canonical Barron measure $\rho$ and threshold $\gamma>0$, define the complex measure
\begin{equation}
d\mu(\theta,a,b,\w)=\frac{-1_{|\theta|\ge\gamma}\:e^{ib\theta}}{{2\pi}\theta^2}\:d\theta\:\times\:a\:d\rho(a,b,\w)
\end{equation}
Then,
\begin{equation}\label{Adefasint}
\psi_{\gamma}^{(\rho)}=\iint \aa_{\theta\w} d\mu(\theta,a,b,\w).
\end{equation}
\end{definition}

\section{Proof of the main theorem}
Variants of the following fact are attributed to Maurey by Pisier \cite{pisier_remarques_1980,Barron1993} and widely used in the literature \cite{Barron1993,e_barron_2022}.
Its statement follows from lemma 1 (page 934) of \cite{Barron1993} when $\psi\neq0$ and is vacuously true when $\psi=0$.
\begin{lemma}[Maurey]
\label{convexcombi}
Let $\mathcal F$ be a subset of a Hilbert space with inner product $\bracket{f}{g}$. Suppose $\Xnorm{f}\le 1$ for all $f\in\mathcal F$, and let 
\begin{equation}
    \psi=\int f d\mu(f),
\end{equation}
where $\mu$ is a complex-valued measure on $\mathcal F$.
Then for each $m\in\NN$ there exists a complex linear combination $\psi_m=\frac1m\sum_{k=1}^ma_k f_k$ of $m$ elements of $\mathcal F$ such that
\[\|\psi_m-\psi\|\le\frac{\|\mu\|}{\sqrt{m}},\]
where $\|\mu\|=\int 1 d|\mu|$ is the total variation of the complex measure $\mu$.
\end{lemma}

Let $\mu$ be the complex measure from \cref{defcomplexmeasure}. Then the absolute value of $\mu$ is
\begin{equation}
d|\mu|(\theta,a,b,\w)=\frac{-1_{|\theta|\ge\gamma}}{{2\pi}\theta^2}\:d\theta\:\times\:|a|\:d\rho(a,b,\w),
\end{equation}
which is a product measure. So its total variation is
\begin{align}\label{TV}
\|\mu\|&=\iint 1d|\mu|(\theta,\w)\\
&=-\frac1{2\pi}\int_{|\theta|\ge\gamma}\frac{1}{\theta^2}\:d\theta\:\times\varphi(\rho)=\frac{\varphi(\rho)}{\pi\gamma}.
\end{align}
We are now ready to finish the proof of the main theorem.
\begin{proof}[Proof of \cref{mainthm}]
Applying the definition of the Barron norm, pick a canonical Barron measure $\rho$ such that $\AS f_\rho=\psi$ and such that
\begin{equation}\label{nearby}
\varphi(\rho)\le(1+\epsilon)\ABnorm{\psi},
\end{equation}
where $\epsilon=2^{-n^2}$. We set the truncation at level $\gamma=\frac1{2\sqrt d}$ as in \cref{AAbound}.
By \cref{prelimit} we can truncate the infra-red part, resulting in an error of order
\begin{align}\label{apply1}
\Xnorm{\psi_\gamma^{(\rho)}-\psi}
&=2^{-\Omega(n^{1+1/d})}\varphi(\rho)
\end{align}
We now approximate $\psi_\gamma^{(\rho)}$ with a finite sum. \cref{Adefasint} decomposes $\psi_\gamma^{(\rho)}$ as an integral over functions $\aa_{\theta\w}$, $\Xnorm{\aa_{\theta\w}}\le1$, against the complex measure $\mu$. We then apply \cref{convexcombi} to obtain a linear combination $\psi_m$ of $m$ terms such that
\begin{equation}\label{applycomplexlemma}
\Xnorm{\psi_m-\psi_\gamma^{(\rho)}}\le\|\mu\|/\sqrt m=C\varphi(\rho)/\sqrt m,
\end{equation}
where $C=2\sqrt d/\pi$. Here, the last equality is by \cref{TV}. Combine \cref{apply1,applycomplexlemma} using the triangle inequality to obtain
\begin{align}
\Xnorm{\psi_m-\psi}
&\le\varphi(\rho)(\tfrac{C}{\sqrt m}+2^{-\Omega(n^{1+1/d})})\\
&\le(1+\epsilon)\ABnorm{\psi}(\tfrac{C}{\sqrt m}+2^{-\Omega(n^{1+1/d})})\\
&=\ABnorm{\psi}(\tfrac{C}{\sqrt m}+2^{-\Omega(n^{1+1/d})})\\
&+\epsilon\ABnorm{\psi}(\tfrac{C}{\sqrt m}+2^{-\Omega(n^{1+1/d})}).
\end{align}
Finally, note that we can absorb the $\epsilon$-term because
\[2^{-n^2}\ABnorm{\psi}(\tfrac{C}{\sqrt m}+2^{-\Omega(n^{1+1/d})})=\ABnorm{\psi}2^{-\Omega(n^{1+1/d})}.\]
\end{proof}

\section{Experiments}    
\cref{fig:thmplot} shows the relation between the anti-symmetric Barron norm $\ABnorm{\psi}$ and the optimal approximation error $\Xnorm{\psi_m-\psi}$ by Slater sums $\psi_m$ as in \cref{mainthm,maincor1}.
Given different target functions $\psi$ we minimized $\Xnorm{\psi_m-\psi}$ over Slater sums to estimate the optimal LHS in \cref{mainthm}. We compared the optimal value with an estimate of the antisymmetric Barron norm $\ABnorm{\psi}$ obtained by a constrained minimization of the network weights over anti-symmetrized neural networks. To construct different target states $\psi$ we used the ground state of a fermionic quantum harmonic oscillator restricted to a sliding window of varying size and location. 
\cref{fig:thmplot} illustrates that the anti-symmetric Barron norm provides an upper bound on the complexity of approximation by determinant-based Ansatz as in \cref{mainthm,maincor1}. 

\begin{figure}[h]
    \centering
    \includegraphics[width=.6\textwidth]{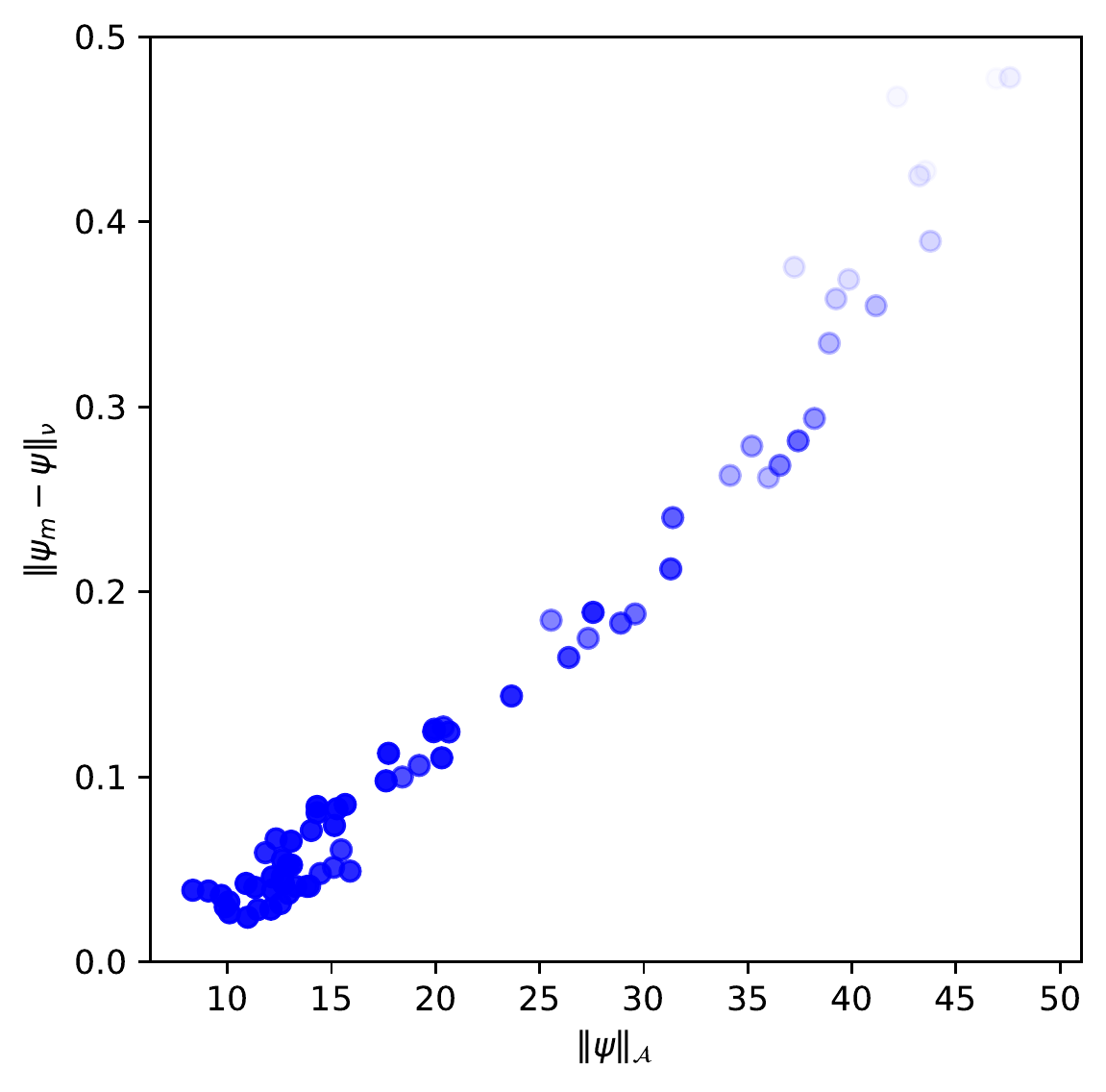}
    \caption{Relation between the anti-symmetric Barron norm $\ABnorm{\psi}$ and the approximation error $\Xnorm{\psi_m-\psi}$ of \cref{mainthm}. Here, $n=6$, $d=3$, and $m=4096$. The opacity is $(1-\epsilon)^{10}$ where $\epsilon$ is the approximation error in the $\epsilon$-smooth anti-symmetric Barron norm $\ABnorm{\psi}^{(\epsilon)}$ (estimates with $\epsilon\approx0.1$ or worse show up as a light color).}
\label{fig:thmplot}
\end{figure}
\subsection{Estimating the anti-symmetric Barron norm}
The Barron norm of \cref{Bspdef} is defined in terms of the ReLU activation, but we can similarly define a Barron norm $\Bnorm{f}^{[\sigma]}$ with any other activation function $\sigma$. Then $\Bnorm{f}=\Bnorm{f}^{[\opn{ReLU}]}$ by definition. We observe that the definition is not overly sensitive to the choice of activation function: 
\begin{lemma}\label{eitherac}
\[\Bnorm{f}^{[\opn{ReLU}]}\le\Bnorm{f}^{[\opn{softplus}]},
\]
where $\opn{softplus}(y)=\log(1+e^y)$. Similarly, $\ABnorm{f}^{[\opn{ReLU}]}$ is bounded by $\ABnorm{f}^{[\opn{softplus}]}$.
\end{lemma}
\begin{proof}
It suffices to show that $\softplus$ is in the closed convex hull of translates of $\relu$. That is, it suffices to write it as a convolution
\begin{equation}
\relu*\phi=\softplus
\end{equation}
for some probability distribution $\psi$ on $\RR$. But we can solve for $\phi$ by differentiating twice:
\begin{align}
\relu'*\phi&=\softplus'=\sigma,\\
\phi=\delta*\phi=\relu''*\phi&=\sigma'=\sigma(1-\sigma).
\end{align}
where $\sigma$ is the logistic sigmoid function. clearly $\phi=\sigma(1-\sigma)$ is positive and $\int\phi=\int\sigma'=1$, so $\softplus$ is in the closed convex hull of translates of $\relu$, 
\end{proof}

\begin{remark}\label{rem:softplussharp}
\textup{Conversely to \cref{eitherac} softplus can approximate ReLU by using weights $a/t,tb,t\w$ with $t\to\infty$ without changing $\varphi(\rho)$. }
\end{remark}

To numerically estimate the Barron norm of a function we make use of \cref{eitherac} and define
\begin{equation}
f_{\a,\b,W}(x)=\sum_{k=1}^ma_k\softplus(\w^{(k)}\cdot x+b_k),
\end{equation}
where $\w^{(k)}\in\mathbb R^{nd}$ and $b_k,a_k\in\mathbb R$ for each $k=1,\ldots,m$. The smoother activation function allows us to use fewer neurons to approximate smooth features, and is justified by \cref{eitherac}. Sharper features can be approximated as in \cref{rem:softplussharp} without changing the norm estimate. We define the corresponding anti-symmetric Ansatz
\begin{equation}
\psi_{\a,\b,W}=\AS f_{\a,\b,W}.
\label{NN}\end{equation}
For functions of this form we have
\begin{align*}
\ABnorm{\psi_{\a,\b,W}}\le\Bnorm{\AS f_{\a,\b,W}}\le&\tilde\varphi(\a,\b,W),\\
\tilde\varphi(\a,\b,W):=&\sum_{k=1}^m|a_k|(\|\w^{(k)}\|_1+b_k).
\end{align*}
We then minimize over $a,b,W\in\RR^{m}\times\RR^{m}\times\RR^{mnd}$ to estimate $\ABnorm{\psi_{\a,\b,W}}$.
Given a target function $\psi$ we define its $\epsilon$-smooth anti-symmetric Barron norm $\ABnorm{\psi}^{(\epsilon)}=\inf\{\ABnorm{\psi'}\:|\:\|\psi'-\psi\|^2\le\epsilon\}$. To estimate $\ABnorm{\psi}^{(\epsilon)}$ we then implement a supervised (SGD) learning procedure for $\psi_{\a,\b,W}$ with a loss function $L$ consisting of two penalties:
\[L(\psi')=\tau_\epsilon(\|\psi_{\a,\b,W}-\psi\|^2)+\lambda\varphi(\a,\b,W),\]
where $\lambda$ is a small constant, and $\tau_\epsilon(y)=\max\{y-\epsilon,0\}$. After the first penalty has converged we then use the optimal value of $\varphi(\a,\b,W)$ as our estimate of $\ABnorm{\psi}^{(\epsilon)}$. 


\section{Conclusion}
We have shown that anti-symmetric functions in the Barron space can be efficiently approximated by determinant-based neural network architectures, and the number of determinants depends on the anti-symmetric Barron norm of the function. Compared to existing bounds for neural network approximations, we obtain a factorially improved error bound. Our result illustrates the importance of choosing an Ansatz which reflects the known symmetries of the problem. It is an open question whether the anti-symmetric Barron norm is a useful characterization of certain challenging quantum states in practice. 

\section*{Acknowledgment}
 (N. A.) was supported by the NSF Quantum Leap Challenge Institute (QLCI) program through grant number OMA-2016245 and by the Simons Foundation under Award No. 825053. This material is also based upon work supported by the U.S. Department of Energy, Office of Science, Office of Advanced Scientific Computing Research and Office of Basic Energy Sciences, Scientific Discovery through Advanced Computing (SciDAC) program (L.L.).  L.L. is a Simons Investigator.



\bibliographystyle{elsarticle-num}
\bibliography{refs}

\appendix

\section{Omitted proofs}
\label{sec:otherproofs}
\begin{proof}[Proof of \cref{ABequiv}]
Let $\PBnorm{\psi}=\inf\{\Bnorm{f}|\AP f=\psi\}$. Then,
\[
\begin{aligned}
\ABnorm{\psi}
&=\inf\{\:\:\Bnorm{f}\:\:|\:\AS f=\psi\}
\\&=\inf\{\:\:\Bnorm{f}\:\:|\:\AP(\sqrt{n!}f)=\psi\}
\\&=\inf\{\:\:\Bnorm{\tfrac1{\sqrt{n!}}g}\:\:|\:\AP(g)=\psi\}
=\tfrac1{\sqrt{n!}}\PBnorm{\psi}.
\end{aligned}
\]
So it suffices to show that $\PBnorm{\psi}=\Bnorm{\psi}$ for any anti-symmetric function $\psi$.

    $\PBnorm{\psi}\le\Bnorm{\psi}$ holds because the infimum $\PBnorm{\psi}=\inf\{\Bnorm{f}|\AP f=\psi\}$ over $f$ includes $f=\psi$.
    
    $\PBnorm{\psi}\ge\Bnorm{\psi}$. To show this, fix $\psi$ and let $f_\rho$ be such that $\AP{f_\rho}=\psi$ and $\varphi(\rho)\to\PBnorm{\psi}$. We need a representation of $\AP f_\rho$ as $f_{\rho'}$ for some measure $\rho'$ to bound its raw Barron norm. But indeed, $\psi=\AP f_\rho=f_{\rho'}$ where
    \[\rho'=\frac1{n!}\sum_{\pi\in S_n}\rho_\pi\quad\rho_\pi(a,b,\w)=\rho((-1)^{\pi}a,b,\pi(\w)).\]
    $\varphi(\rho_\pi)=\varphi(\rho)$ for each $\pi$, so the same holds for $\rho'$. Now $\Bnorm{\psi}\le\varphi(\rho')=\varphi(\rho)\to\PBnorm{\psi}$ which proves the inequality.

\end{proof}

\subsection{The Fourier inversion formula for ReLU}
\label{invformula}
We verify \cref{Fdef} for the ReLU activation with $\widehat\activation$ given by \cref{Frelu}. \cref{Fdef} claims that $\LP\activation\gamma=p_\gamma+O(\gamma g)$ where $p$ is a low-degree polynomial and $g$ is bounded by a polynomial. Here we have defined $\LP\activation\gamma=\activation-\HP\activation\gamma$.
    We first evaluate the high-pass part
\[\HP\activation{\gamma}(y):=
\frac1{\sqrt{2\pi}}\int_{|\theta|>\gamma}\hat\sigma(\theta)e^{i\theta y}d\theta
=|y|/2-\frac{\cos(\gamma y)}{\pi \gamma }-\frac{y\opn{Si}(\gamma y)}\pi,\]
where $\opn{Si}(y)=\int_0^y\frac{\sin s}sds$. Since $\opn{ReLU}(y)=|y|/2+y/2$,
\begin{align}
\LP\activation{\gamma }(y):=\activation(y)-\HP\activation\gamma(y)
&=y/2+\frac{\cos(\gamma y)}{\pi \gamma }+\frac{y\opn{Si}(\gamma y)}\pi.
\end{align}
Write $\LP\activation\gamma=p_\gamma+\varepsilon$ where
\[p_\gamma(y)=y/2+\frac1{\pi\gamma}.\]
Then the remainder satisfies
\begin{equation}\label{boundlp}
|\varepsilon|\le|\frac{\cos(\gamma y)-1}{\pi \gamma }|+|\frac{y\opn{Si}(\gamma y)}\pi|\le\frac{(\gamma y)^2}{2\pi \gamma }+\frac{y\cdot(\gamma y)}{\pi}=\gamma g(y),\qquad g(y):=\frac{3 }{2\pi}y^2.\end{equation}

\subsection{Infra-red estimate}
\label{ests}

\begin{proof}[Proof of \cref{AAbound}]
By \cref{Frelu} we have $|\widehat\relu(\theta)|=\frac1{\sqrt{2\pi}\theta^2}$ for $|\theta|>0$, so \cref{limeps} implies
\begin{equation}\label{intwp}
\Xnorm{A_{\w,b;\gamma}-A_{\w,b}}\vphantom{\frac11}\le
\frac1{2\pi}\int_{|\theta|<\gamma}\frac{\Xnorm{\aa_{\theta\w}}}{\theta^2}d\theta,
\end{equation}
Now apply \cref{Dttlowbound} to get the bound
\begin{equation}\label{sep}
\Xnorm{\aa_{\theta\w}}^2\le\Big(\frac{|\theta|}{2\gamma}\Big)^2\Big(\frac{|\theta|}{2\gamma}\Big)^{\Omega(n^{1+1/d})}\quad \text{for }|\theta|\le\gamma.
\end{equation}
Here we have taken out two powers of $\theta/(2\gamma)$ in order to cancel the $1/\theta^2$ in \cref{intwp}. Rearranging \cref{sep} yields
\begin{equation}\label{loose}
\frac{\Xnorm{\aa_{\theta\w}}}{\theta^2}\le\gamma^2\Big(\frac{|\theta|}{2\gamma}\Big)^{\Omega(n^{1+1/d})}\le\gamma^2 2^{-\Omega(n^{1+1/d})}.
\end{equation}
Now substitute \cref{loose} into \cref{intwp} and use $\gamma=O(1)$ to obtain the result.
\end{proof}

\end{document}